\documentclass[12pt, reqno]{amsart}
\usepackage{amsmath, amsthm, amscd, amsfonts, amssymb, graphicx, color}
\usepackage[bookmarksnumbered, colorlinks, plainpages]{hyperref}
\hypersetup{colorlinks=true,linkcolor=red, anchorcolor=green, citecolor=cyan, urlcolor=red, filecolor=magenta, pdftoolbar=true}

\newtheorem{theorem}{Theorem}[section]
\newtheorem{lemma}[theorem]{Lemma}
\newtheorem{proposition}[theorem]{Proposition}
\newtheorem{cor}[theorem]{Corollary}
\newtheorem{definition}[theorem]{Definition}
\newtheorem{example}[theorem]{Example}

\theoremstyle{remark}
\newtheorem{remark}[theorem]{\bf{Remark}}
\numberwithin{equation}{section}
\begin{document}

\title [A note on the $A$-numerical range of semi-Hilbertian operators]{\Small{A note on the $A$-numerical range of semi-Hilbertian operators}}
	\author[A. Sen, R. Birbonshi and K. Paul]{Anirban Sen, Riddhick Birbonshi and Kallol Paul}
	
\address[Sen] {Department of Mathematics, Jadavpur University, Kolkata 700032, West Bengal, India}
\email{anirbansenfulia@gmail.com}

\address[Birbonshi] {Department of Mathematics, Jadavpur University, Kolkata 700032, West Bengal, India}
\email{riddhick.math@gmail.com}

\address[Paul] {Department of Mathematics, Jadavpur University, Kolkata 700032, West Bengal, India}
\email{kalloldada@gmail.com}

\thanks{Mr. Anirban Sen would like to thank CSIR, Govt. of India, for the financial
	support in the form of Senior Research Fellowship under the mentorship of
	Prof. Kallol Paul}



\maketitle	

\begin{abstract} 
	In this paper we explore the relation between the $A$-numerical range and the $A$-spectrum of $A$-bounded operators in the setting of semi-Hilbertian structure. We introduce a new definition of $A$-normal operator and prove that closure of the $A$-numerical range of an $A$-normal operator is the convex hull of the $A$-spectrum. We further prove Anderson's theorem for the sum of $A$-normal and $A$-compact operators which improves and generalizes the existing result on Anderson's theorem for $A$-compact operators. Finally we introduce strongly $A$-numerically closed class of operators and along with other results prove that the class of $A$-normal operators is strongly $A$-numerically closed.
\end{abstract}

\section{Introduction and preliminaries}


Let $\mathcal{H}$ be a complex separable infinite dimensional Hilbert space endowed with the inner product $\langle \cdot, \cdot \rangle$ and let $\|\cdot\|$ be the norm induced by the inner product. Let $\mathcal{B}(\mathcal{H})$ and $\mathcal{K}(\mathcal{H})$ be the algebra of all bounded linear operators and compact linear operators on $\mathcal{H},$ respectively.  Throughout this article, for a subset $\Delta$ of the complex field $\mathbb{C}$, $conv(\Delta)$ stands for the convex hull of the subset $\Delta$ and $\mathbb{D}$ stands for the open unit disk, i.e., $\mathbb{D}=\{z \in \mathbb{C} : |z|<1\}.$
For any $T \in \mathcal{B}(\mathcal{H}),$ the range and the null space of $T$ are denoted by $R(T)$ and $N(T),$ respectively. The closure of $R(T)$ with respect to the usual norm of $\mathcal{H}$ is denoted by $\overline{R(T)}.$
An operator $A \in \mathcal{B}(\mathcal{H})$ is called positive if $\langle Ax,x\rangle \geq 0$ for all $x \in \mathcal{H}$ and we use the notation $A \geq 0$.
Henceforth we reserve the letter $A$  for a non-zero positive operator on $\mathcal{H}.$  The orthogonal projection on $\overline{R(A)}$  is denoted by  $P_{\overline{R(A)}}.$  
Naturally, $A$ generates a semi-inner product $\langle \cdot, \cdot \rangle_A$ on $\mathcal{H}$ given by $\langle x,y \rangle_A=\langle Ax,y \rangle$ for all $x,y \in \mathcal{H}.$ The semi-norm induced by the semi-inner product is given by $\|x\|_A=\sqrt{\langle x,x \rangle_A}$ for all $x \in \mathcal{H}.$ The vector space $\mathcal{H}$ endowed with the semi-inner product $\langle \cdot, \cdot \rangle_A$ is called a semi-Hilbertian space.  
The semi-inner product $\langle \cdot, \cdot \rangle_A$ induces an inner product $ \left[.\, , \, .\right] $ on the quotient space $\mathcal{H}/N(A)$  defined by  $\left[\overline{x}, \overline{y} \right]=\langle x, y \rangle_A$ for all $\overline{x}, \overline{y} \in \mathcal{H}/N(A).$ The completion of
$\left(\mathcal{H}/N(A),[\cdot,\cdot]\right)$ is isometrically
isomorphic to the Hilbert space $\textbf{R}(A^{1/2})=\left(R(A^{1/2}), (\cdot, \cdot)\right),$ where the inner product $(\cdot, \cdot)$ is given by $$(A^{1/2}x, A^{1/2}y)=\langle P_{\overline{R(A)}}x, P_{\overline{R(A)}}y \rangle$$ for all $x,y \in \mathcal{H}$ \cite{branges1,feki_LAA_2020}.
Clearly  $R(A^{1/2})\subseteq \overline{R(A)}$ and so for all $x,y \in \mathcal{H}$ this leads to the following useful relations:
\begin{align*}
	(Ax, Ay)=\langle x, y \rangle_{A}~~\mbox{and}~~\|Ax\|_{\textbf{R}(A^{1/2})}=\|x\|_{A}.
\end{align*}
Moreover, $R(A)$ is dense in $\textbf{R}(A^{1/2}),$ for further readings concerning the Hilbert space $\textbf{R}(A^{1/2})$ we refer to \cite{arias 1}. For a given $T \in 
\mathcal{B}(\mathcal{H}),$ if there exists a constant $c >0$ such that $ \|Tx\|_A \leq c \|x\|_A $ for all $ x \in \overline{R(A)}$ then $A$-operator semi-norm of $T$, denoted by $\|T\|_A,$ is defined as 
\begin{align*}
	\|T\|_A=\sup_{x \in \overline{R(A)}, x \neq 0}\frac{\|Tx\|_A}{\|x\|_A}.
\end{align*}
Given $T \in \mathcal{B}(\mathcal{H}),$ an operator $W \in \mathcal{B}(\mathcal{H})$ is called an $A$-adjoint of $T$ if $\langle Tx,y \rangle_A=\langle x,Wy \rangle_A$ for all $x,y \in \mathcal{H}.$ By Douglas theorem \cite{dog1}, 
the set of all operators which admit an $A$-adjoint is denoted by $\mathcal{B}_A(\mathcal{H}),$ and is given by
$$\mathcal{B}_A(\mathcal{H})=\{T\in \mathcal{B}(\mathcal{H}) :R(T^*A) \subseteq R(A) \}.$$
Douglas theorem ensures that if $T \in \mathcal{B}_A(\mathcal{H})$ then the equation $AX=T^*A$ has a unique solution which is denoted by $T^{\sharp_A}$ and satisfies $R(T^{\sharp_A})\subseteq \overline{R(A)}.$ Here $T^{\sharp_A}=A^{\dagger}T^*A,$ where $A^{\dagger}$ is the Moore-Penrose \cite{Engl_JMAA_1981,Penr55} inverse of $A$. 
 The set of all operators admitting $A^{1/2}$-adjoint  is denoted by $\mathcal{B}_{A^{1/2}}(\mathcal{H}).$ Again by Douglas theorem
\begin{eqnarray*}
	\mathcal{B}_{A^{1/2}}(\mathcal{H})=\{T\in \mathcal{B}(\mathcal{H}) :\exists ~ c>0 ~\mbox{such that}~ \|Tx\|_A\leq c\|x\|_A~~ \forall x 
	\in \mathcal{H}\}.
\end{eqnarray*}
 An operator $T \in \mathcal{B}_{A^{1/2}}(\mathcal{H})$ is known as $A$-bounded operator and if $T \in \mathcal{B}_{A^{1/2}}(\mathcal{H})$ then $T(N(A))\subseteq N(A).$ For more details about $\mathcal{B}_{A}(\mathcal{H})$ and $\mathcal{B}_{A^{1/2}}(\mathcal{H})$ we refer  readers to \cite{arias 1,a2,a3,feki_AOFA_2020}.
 
The notion of $A$-compact operators generalizing the compact operators was introduced and studied in \cite{BKPS_LAA_23}. Let $\mathcal{K}_{A^{1/2}}(\mathcal{H})$ 
denote the collection of all $A$-compact operators. An operator $T\in \mathcal{B}_{A^{1/2}}(\mathcal{H})$ is said to be $A$-invertible in $\mathcal{B}_{A^{1/2}}(\mathcal{H})$ if there exists $S\in  \mathcal{B}_{A^{1/2}}(\mathcal{H})$ such that $ATS=AST=A$ and $S$ is called $A$-inverse of $T,$ see \cite{baklouti_BJMA_2022}. 
For $T \in \mathcal{B}_{{A}^{1/2}}(\mathcal{H}),$ the $A$-resolvent set and the $A$-spectrum of $T,$  denoted respectively by $\rho_A(T)$ and $\sigma_A(T),$  are defined as
 $$\rho_A(T)=\{\lambda \in \mathbb{C} : (T-\lambda I) ~~\mbox{is $A$-invertible in}~~ \mathcal{B}_{{A}^{1/2}}(\mathcal{H})\}$$
  and 
  $$\sigma_A(T)=\{\lambda \in \mathbb{C} : (T-\lambda I) ~~\mbox{is not $A$-invertible in}~~ \mathcal{B}_{{A}^{1/2}}(\mathcal{H})\}.$$
  The $A$-spectral radius of $T$  (see \cite{Saddi}) is defined as 
\begin{align}\label{r1}
r_{A}(T)=\lim_{n \to \infty}\|T^n\|^{1/n}_{A}.
\end{align}
In  \cite{feki_BJOM_2022}, it is proved that 
\begin{align}\label{r2}
r_{A}(T)=\sup\{|\lambda| : \lambda \in \sigma_A(T)\}.
\end{align}
For $T \in \mathcal{B}_{A^{1/2}}(\mathcal{H}),$ the $A$-numerical
range and $A$-numerical radius of $T,$ denoted respectively by $W_{A}(T)$ and $w_A(T),$  are defined as 
 $$W_{A}(T)=\{ \langle Tx,x \rangle_{A} :  x \in \mathcal{H}, \|x\|_{A}=1 \}$$
 and
 $$w_{A}(T)=\sup \{ |\lambda| :  \lambda \in W_{A}(T) \}.$$
Clearly $w_A(T)\leq \|T\|_A$. It is well known that for any $T \in \mathcal{B}_{A^{1/2}}(\mathcal{H}),$ the $A$-numerical range of $T$ is a convex subset of $\mathbb{C}$ (see \cite[Th. 2.1]{bak}). 
The $A$-numerical radius and $A$-spectral radius inequalities have been studied by many mathematicians  \cite{BKPS_LAA_23,BNK_RM_2021,feki_BJOM_2022,KZ_JCAM_2023,MXZ_LAA_2020,zamani} over the years. To the best of our knowledge,  there is hardly any literature available on the $A$-numerical range of semi-Hilbertian operators. The main purpose of this article is to study the $A$-numerical range of semi-Hilbertian operators. We also study the $A$-numerical range and $A$-spectrum of $A$-normal operators, the definition of which is introduced here in a slightly different way.
Following \cite{arias 1}, an operator $T \in \mathcal{B}_{A}(\mathcal{H})$ was defined as  $A$-normal if $T^{\sharp_A}T=TT^{\sharp_A}.$  However, we here propose to refine the definition of $A$-normal operator as follows: 
\begin{definition}\label{d1}
	An operator $T \in \mathcal{B}_{A}(\mathcal{H})$ is said to be $A$-normal if
	 \[ AT^{\sharp_A}T=ATT^{\sharp_A} .\]
	\end{definition}\label{new-normal}
	The justification for refining the definition of $A$-normal operator will be  clear from the relevant results developed in this paper in due course of time. 
	 Clearly, if $T$ is an $A$-normal operator in the sense of \cite{arias 1} then $T$ is $A$-normal in the sense of Definition 1.1.  However, the converse is not necessarily true. For example, consider the matrices $A= \begin{pmatrix}
	1&1\\
	1&1
\end{pmatrix}$ and $T=\begin{pmatrix}
	2&2\\
	0&0
\end{pmatrix}.$ Then  we get,  $TT^{\sharp_A}=\begin{pmatrix}
	4&4\\
	0&0
\end{pmatrix} \neq 
\begin{pmatrix}
	2&2\\
	2&2
\end{pmatrix}=T^{\sharp_A}T$ but 
$ATT^{\sharp_A}=AT^{\sharp_A}T=\begin{pmatrix}
	4&4\\
	4&4
\end{pmatrix}.$  Observe that the converse is true if $T$ satisfies the relation $R(TT^{\sharp_A})\subseteq \overline{R(A)}.$  Thus the class of $A$-normal operators  introduced here is larger than that defined in \cite{arias 1}. Henceforth, we consider $A$-normal operator as per  Definition \ref{d1}. 

We prove that if $T$ is an $A$-normal operator then 
$\overline{W_A(T)}=conv(\sigma_A(T)),$ which generalizes the classical 
result  that the closure of the numerical range of a normal operator is the convex hull of its spectrum \cite{Berberian_DJM_1981}. We prove that $T$ is $A$-normal if and only if $ \|Tx\|_A = \|T^{\sharp A}x \|_A$ for all $ x \in \mathcal{H},$ which again generalizes the analogous result for normal operators. 
 A new version of Anderson's theorem for semi-Hilbertian operators is proved in Section \ref{s2} which generalizes both \cite[Th. 2.7]{BKPS_LAA_23} and \cite[Th. 4]{BSS_JMA_2018}.
In Section \ref{s3}, we investigate which classes of operators (say $\mathcal{C}$) in $\mathcal{B}_{A^{1/2}}(\mathcal{H})$ has the property that for any $T \in \mathcal{C}$ and $\epsilon>0$ there exists $K \in \mathcal{K}_{{A}^{1/2}}(\mathcal{H})$ with $\|K\|_A<\epsilon$ such that $T+K \in \mathcal{C}$ and $\overline{W_A(T+K)}=W_A(T+K).$ We call such an operator  $T \in \mathcal{C}$ as $A$-strongly numerically closed. Moreover, we prove that if $R(A)$ is closed then  $\mathcal{B}_{A^{1/2}}(\mathcal{H})$ and some of its sub-classes are strongly $A$-numerically closed. Our results can be considered as  generalizations of \cite[Prop. 1.3]{Bourin_JOT_2003}, \cite[Th. 1.3]{liang_JMAA_2023} and \cite[Prop. 5.1]{zhu_BJMA_2015}. Finally, we construct an example to justify that the closedness of $R(A)$ is not a necessary condition.  

\section{$A$-spectrum and $A$-numerical range}\label{s1}

We begin this section with the following lemma which provide important connections between the elements of $\mathcal{B}_{A^{1/2}}(\mathcal{H})$ and $\mathcal{B}(\textbf{R}(A^{1/2})).$ The operator $Z_A:\mathcal{H}\to \textbf{R}(A^{1/2}),$ defined by $Z_Ax=Ax$ for all  $x\in \mathcal{H},$ plays an important role in this interplay. 

\begin{lemma}\cite{arias 1,feki_LAA_2020}\label{lemma1-2}
(i) 	Let $T \in \mathcal{B}(\mathcal{H}).$ Then $T \in \mathcal{B}_{A^{1/2}}(\mathcal{H})$ if and only if there exists a unique $\widetilde{T} \in \mathcal{B}(\textbf{R}(A^{1/2}))$ such that $Z_AT=\widetilde{T}Z_A.$ 
Moreover, if $T,S \in \mathcal{B}_{A^{1/2}}(\mathcal{H})$ and $\lambda \in \mathbb{C}$ then $\widetilde{T+\lambda S}=\widetilde{T}+\lambda \widetilde{S}$ and $\widetilde{TS}=\widetilde{T}\widetilde{S}.$ \\

(ii)  Let  $\widetilde{T} \in \mathcal{B}(\textbf{R}(A^{1/2})).$  Then there exists $T \in \mathcal{B}(\mathcal{H})$ such that $Z_AT=\widetilde{T}Z_A$ if and only if $R(\widetilde{T}Z_A) \subseteq R(Z_A)=R(A).$ In such case there exists a unique $T \in \mathcal{B}_{A^{1/2}}(\mathcal{H})$ such that $R(T) \subseteq \overline{R(A)}.$
	
		Furthermore, if $R(A)$ is closed then given any $\widetilde{T} \in \mathcal{B}(\textbf{R}(A^{1/2}))$ there exists a unique $T \in \mathcal{B}_{A^{1/2}}(\mathcal{H})$ with $R(T) \subseteq {R(A)}$ such that $Z_AT=\widetilde{T}Z_A.$ 
\end{lemma}

	
The next lemma provides some useful results.
\begin{lemma}\label{lemma3-6}
	\begin{enumerate}
\item [(i)]	Let $T \in \mathcal{B}_A(\mathcal{H}).$ Then $\widetilde{T^{\sharp_A}}=(\widetilde{T})^*$ (see \cite{maj_LAMA}). 
\item [(ii)]		Let $T \in \mathcal{B}_{A^{1/2}}(\mathcal{H}).$ Then 
$\|T\|_A=\|\widetilde{T}\|_{\mathcal{B}(\textbf{R}(A^{1/2}))}$ (see \cite{feki_AOFA_2020}). 
\item [(iii)]		Let $T \in \mathcal{B}_{A^{1/2}}(\mathcal{H})$ and $R(A)$ be closed. Then $T$ is $A$-compact if and only if $\widetilde{T}$ is compact (see \cite{BKPS_LAA_23}).
\item [(iv)]   Let $T \in \mathcal{B}_{A^{1/2}}(\mathcal{H}).$ Then 
\begin{eqnarray*}
	W_A(T) \subseteq W(\widetilde{T}) \subseteq \overline{W_A(T)}.
\end{eqnarray*}
Furthermore, if $R(A)$ is closed then $W_A(T) = W(\widetilde{T})$ (see \cite{BKPS_LAA_23}). 

\end{enumerate}
\end{lemma}

We next introduce the definition of $A$-point spectrum and $A$-approximate spectrum in  the semi-Hilbertian structure and then  study their properties.

\begin{definition}
	Let $T \in \mathcal{B}_{{A}^{1/2}}(\mathcal{H}).$  The $A$-point spectrum of $T,$  denoted by $\sigma_{A_p}(T),$ is defined as 
	$$\sigma_{A_p}(T)=\{\lambda \in \mathbb{C} : \exists ~~ x(\neq 0) \in \overline{R(A)}~~\mbox{such that}~~P_{\overline{R(A)}}Tx=\lambda x\}.$$
\end{definition}
Observe that  the Moore-Penrose inverse satisfies the following relations (see\cite{Engl_JMAA_1981})
\[ A^{\dagger}A=(A^{1/2})^{\dagger}A^{1/2}=P_{\overline{R(A)}}, \, AP_{\overline{R(A)}}=A \, \mbox{and} \,A^{1/2}P_{\overline{R(A)}}=P_{\overline{R(A)}}A^{1/2}=A^{1/2}\] 
and so
the $A$-point spectrum of $T$ can also be expressed as
\begin{align*}
	\sigma_{A_p}(T)=&\{\lambda \in \mathbb{C} : \exists ~~ x(\neq 0) \in \overline{R(A)}~~\mbox{such that}~~ATx=\lambda Ax\}\\
	= &\{\lambda \in \mathbb{C} : \exists ~~ x(\neq 0) \in \overline{R(A)}~~\mbox{such that}~~A^{1/2}Tx=\lambda A^{1/2}x\}.
\end{align*}

\begin{definition}
	Let $T \in \mathcal{B}_{{A}^{1/2}}(\mathcal{H}).$ The $A$-approximate point spectrum of $T$, denoted by $\sigma_{A_{app}}(T),$  is defined as 
	$$\sigma_{A_{app}}(T)=\{\lambda \in \mathbb{C} : \exists ~~ \{x_n\} \subseteq \overline{R(A)}~~\mbox{with}~~\|x_n\|_A=1~~\mbox{such that}~~\|(T-\lambda I)x_n\|_A \to 0 \}.$$
\end{definition}
Clearly,  $\sigma_{A_p}(T) \subseteq \sigma_{A_{app}}(T).$
We next prove the following propositions which explore the  relation between the various notions of $A$-spectrum and $A$-point  spectrum, $A$-approximate point spectrum and $A$-numerical range in the semi-Hilbertian structure.

\begin{proposition} \label{P3} Let $T \in \mathcal{B}_{{A}^{1/2}}(\mathcal{H}).$ Then  \\
(i)  $\sigma_{A_p}(T) \subseteq \sigma_{A}(T).$\\
(ii)  $\sigma_{A_p}(T) \subseteq \sigma_p(\widetilde{T})$ and if $R(A)$ is closed then $\sigma_{A_{p}}(T)=\sigma_p(\widetilde{T}).$\\
(iii)  $\sigma_{A_{app}}(T) \subseteq \sigma_{app}(\widetilde{T})$ and if $R(A)$ is closed then $\sigma_{A_{app}}(T)=\sigma_{app}(\widetilde{T}).$
\end{proposition}

\begin{proof}
(i) 	Let $\lambda \in \sigma_{A_p}(T).$ Then there exists $x(\neq 0)\in \overline{R(A)}$ such that $ATx=\lambda Ax,$ i.e., $(T-\lambda I)x \in N(A).$ If possible let $\lambda \notin \sigma_{A}(T).$ Therefore, there exists $S \in \mathcal{B}_{{A}^{1/2}}(\mathcal{H})$ such that $A(T-\lambda I)S=AS(T-\lambda I)=A,$ this implies that $A(T-\lambda I)Sx=AS(T-\lambda I)x=Ax.$ As $N(A)$ is invariant under $S$ so $AS(T-\lambda I)x=0.$ Thus, we obtain $Ax=0,$ i.e., $x\in N(A)$ which contradicts the fact that $x(\neq 0)\in \overline{R(A)}.$\\
(ii)  Let $\lambda \in \sigma_{A_p}(T).$ Then from the definition there exists $x(\neq 0) \in \overline{R(A)}$ such that 
$ATx=\lambda Ax,$ this implies that $\widetilde{T}Ax=\lambda Ax$ and thus $\lambda \in \sigma_p(\widetilde{T})$.

If $R(A)$ is closed then $R(A^{1/2})=R(A).$ Thus if $\lambda \in \sigma_p(\widetilde{T})$ then there exits $x(\neq 0) \in {R(A)}$ such that $\widetilde{T}Ax=\lambda Ax.$ This implies that $ATx=\lambda Ax,$ i.e., $\lambda \in \sigma_{A_p}(T).$ \\
(iii)  Suppose that $\lambda \in \sigma_{A_{app}}(T).$ Then there exists $\{x_n\} \subseteq \overline{R(A)}$ with  $\|x_n\|_A=1$ such that $\|(T-\lambda I)x_n\|_A \to 0.$ Therefore, we have $\|Ax_n\|_{\textbf{R}(A^{1/2})}=1$ and $\|A(T-\lambda I)x_n\|_{\textbf{R}(A^{1/2})} \to 0.$ Now, by applying Lemma \ref{lemma1-2} (i) we get $\|(\widetilde{T}-\lambda \widetilde{I})Ax_n\|_{\textbf{R}(A^{1/2})} \to 0.$ Thus we have $\lambda \in \sigma_{app}(\widetilde{T}).$

Now, if $R(A)$ is closed and $\lambda \in \sigma_{app}(\widetilde{T})$ then there exists $\{x_n\} \subseteq R(A)$ with $\|Ax_n\|_{\textbf{R}(A^{1/2})}=1$ such that $\|(\widetilde{T}-\lambda \widetilde{I})Ax_n\|_{\textbf{R}(A^{1/2})} \to 0.$ Therefore, Lemma \ref{lemma1-2} (i) implies that $\|A(T-\lambda I)x_n\|_{\textbf{R}(A^{1/2})} \to 0,$ and thus $\|(T-\lambda I)x_n\|_A \to 0$ with $\|x_n\|_A=1.$ Hence, $\lambda \in \sigma_{A_{app}}(T)$ and this completes the proof of (iii).
\end{proof}

Next we give an example to show that the closedness of $R(A)$ is not necessary in Proposition \ref{P3}.

\begin{example}\label{ex1}
	Let $\{e_n\}$ denotes the standard orthonormal basis of the Hilbert space $\ell_2.$
	Let $A : \ell_2 \to \ell_2$ be the bounded linear operator defined as 
	$$Ae_n= \frac{1}{n}e_n, \,\forall n \in \mathbb{N}.$$
 Then $A$ is a positive operator. 
	Let $\{\lambda_n = e^{i \theta_n}\}$ be a sequence of scalars with $\theta_n= \pi/4+\pi/{8n}$ for all $ n \in \mathbb{N}.$ Then $\lambda_n \to \lambda_0.$  Consider the bounded linear operator $ T : \ell_2 \to \ell_2$ defined as 
	$$Te_n=\lambda_n e_n, \,\forall n \in \mathbb{N}.$$
 Then it is easy to verify that 
	 $\sigma_{A_p}(T)=\{\lambda_n : n \in \mathbb{N}\}.$ Again 
	$$\|(T-\lambda_0I)\sqrt{n}e_n\|_A=|\lambda_n-\lambda_0|\to 0,$$ as $n \to \infty$ so $\lambda_0 \in \sigma_{A_{app}}(T).$ Suppose that $\mu \notin \{\lambda_n : n \in \mathbb{N}\cup\{0\}\}.$ Then there exists $\delta>0$ such that $|\mu-\lambda_n|\geq \delta,\, \forall n \in \mathbb{N}\cup\{0\}.$ Let $\{\widetilde{x_i}=\sum_{n=1}^{\infty}x_n^{(i)}e_n\}$ be a sequence in $\ell^2$ with $\|\widetilde{x_i}\|_A=\left(\sum_{n=1}^{\infty}|x_n^{(i)}|^2\|e_n\|^2_A\right)^{1/2}=1.$ Thus, we obtain
	$$\|(T-\mu I)\widetilde{x_i}\|^2_A=\sum_{n=1}^{\infty}|\lambda_n-\mu|^2|x_n^{(i)}|^2\|e_n\|^2_A\geq \delta^2>0.$$
	As $\{\widetilde{x_i}\}$ is an arbitrary sequence in $\ell_2$ hence $\mu \notin \sigma_{A_{app}}(T).$ Therefore, $\sigma_{A_{app}}(T)=\{\lambda_n : n \in \mathbb{N}\cup\{0\}\}.$ Observe that  $\sigma_A(T)$ is a compact subset of $\mathbb{C}$ with countable boundary as because $\partial\sigma_A(T)\subseteq \sigma_{A_{app}}(T).$  Hence $\sigma_A(T)=\partial\sigma_A(T)=\{\lambda_n : n \in \mathbb{N}\cup\{0\}\}.$ Now, by applying Lemma \ref{lemma1-2} (i), we get $\widetilde{T} \in \mathcal{B}(\textbf{R}(A^{1/2}))$ as 
	$$\widetilde{T}e_n= \lambda_ne_n, \,\forall n \in \mathbb{N}.$$  Similarly, we obtain that $\sigma_p(\widetilde{T} )=\{\lambda_n : n \in \mathbb{N}\}$ and $\sigma(\widetilde{T} )=\sigma_{app}(\widetilde{T} )=\{\lambda_n : n \in \mathbb{N}\cup\{0\}\}.$ Thus we get $\sigma_{A_{p}}(T)=\sigma_p(\widetilde{T})$ and $\sigma_{A_{app}}(T)=\sigma_{app}(\widetilde{T}),$ though  $R(A)$ is not closed.  
\end{example}

\begin{proposition}\label{Prop1}
	Let $T \in \mathcal{B}_{{A}^{1/2}}(\mathcal{H}).$ Then  $\sigma_{A_p}(T) \subseteq W_A(T)$ and $\sigma_{A_{app}}(T) \subseteq \overline{W_A(T)}.$ 
\end{proposition}

\begin{proof}
	First we consider any $\lambda \in \sigma_{A_p}(T).$ Then there exists $x \in \overline{R(A)}$ with $\|x\|_A=1$ such that $A^{1/2}Tx=\lambda A^{1/2}x.$ Thus we have $\lambda = \langle Tx,x\rangle_A \in W_A(T),$ as desired.
	
	To prove the second assertion, let $\lambda \in \sigma_{A_{app}}(T).$ Then there exists $\{x_n\} \subseteq \overline{R(A)}$ with  $\|x_n\|_A=1$ such that $\|(T-\lambda I)x_n\|_A \to 0$ as $n \to \infty.$ Therefore, we have 
	\begin{align*}
		|\langle Tx_n,x_n \rangle_A- \lambda| =|\langle (T-\lambda I)x_n,x_n \rangle_A| \leq \|(T-\lambda I)x_n\|_A \to 0.
	\end{align*}
This implies that $\lambda \in \overline{W_A(T)},$ as desired.
\end{proof}

In the next theorem we establish a relation between the boundary of $A$-spectrum and $A$-approximate spectrum of $A$-bounded operators. To prove this result we need the following lemmas which are proved in \cite{baklouti_BJMA_2022}.

\begin{lemma}\label{LL1}
		Let $T \in \mathcal{B}_{A^{1/2}}(\mathcal{H}).$ Then the following conditions are equivalent:\\
		(i)~~$TP_{\overline{R(A)}}$ is $A$-invertible in $\mathcal{B}_{A^{1/2}}(\mathcal{H}).$\\
		(ii)~~$P_{\overline{R(A)}}T$ is $A$-invertible in $\mathcal{B}_{A^{1/2}}(\mathcal{H}).$\\
		(iii)~~$T$ is $A$-invertible in $\mathcal{B}_{A^{1/2}}(\mathcal{H}).$
\end{lemma}

\begin{lemma}\label{lem31}
	Let $T_1,T_2\in \mathcal{B}_{A^{1/2}}(\mathcal{H})$ be such that $T_1T_2=T_2T_1.$ Then $T_1T_2$ is $A$-invertible in $\mathcal{B}_{A^{1/2}}(\mathcal{H})$ if and only if $T_1$ and $T_2$ are $A$-invertible in $\mathcal{B}_{A^{1/2}}(\mathcal{H}).$
\end{lemma}

\begin{lemma}\label{LL2}
	Let $T \in \mathcal{B}_{A^{1/2}}(\mathcal{H}).$ Then for all $\lambda >\|T\|_A,$ the operator $(\lambda I-T)$ is $A$-invertible in $\mathcal{B}_{A^{1/2}}(\mathcal{H}).$
\end{lemma}

\begin{theorem}\label{Th1}
	Let $T \in \mathcal{B}_{A^{1/2}}(\mathcal{H}).$ Then the boundary  of the $A$-spectrum of $T$ is contained in the $A$-approximate spectrum of $T,$ i.e., $\partial\sigma_{A}(T) \subseteq \sigma_{A_{app}}(T).$
\end{theorem}

\begin{proof}
 If possible let $\lambda \in \partial\sigma_{A}(T)$  but $\lambda \notin \sigma_{A_{app}}(T).$  There exists a sequence $\{\mu_n\}$ in $\rho_A(T)$ such that $\mu_n \to \lambda$ as $n \to \infty.$   Also there exists $c>0$ such that 
	\begin{align*}
		\|(T-\lambda I)x\|_A\geq c\|x\|_A ~~\mbox{for all $x \in \mathcal{H}.$}
	\end{align*} 
Now, we have
\begin{align}\label{Th1e1}
	P_{\overline{R(A)}}(T-\lambda I)=P_{\overline{R(A)}}\left(I-(\lambda -\mu_n)S_{\mu_n}\right)(T-\mu_n I),
\end{align}
where  for each $n,$ $S_{\mu_n} \in \mathcal{B}_{A^{1/2}}(\mathcal{H})$ is the $A$-inverse of $(T-\mu_n I).$ Since $\sigma_A(T)$ is a closed subset of $\mathbb{C},$ so $\lambda \in \partial\sigma_{A}(T)$ implies that $\lambda \in \sigma_{A}(T).$ To arrive at a contradiction our aim is to show that $\lambda \in \rho_A(T)$. Now, by Lemma \ref{LL1} and Lemma \ref{lem31} from  \eqref{Th1e1} we obtain this is equivalent to show that there exists $n\in \mathbb{N}$ such that $\left(I-(\lambda -\mu_n)S_{\mu_n}\right)$ is $A$-invertible in $\mathcal{B}_{A^{1/2}}(\mathcal{H}).$ As $\mu_n \to \lambda,$ there exists $n_0\in \mathbb{N}$ such that $|\lambda-\mu_n|<c/2$ for all $n \geq n_0.$ Therefore for all $x \in \mathcal{H},$ we get
\begin{align}\label{Th1e2}
	\|(T-\mu_n)x\|_A&=\|(T-\lambda I)x-(\mu_n-\lambda)x\|_A\nonumber\\
	&\geq \left|\|(T-\lambda I)x\|_A-|\mu_n-\lambda|\|x\|_A\right|\nonumber\\
	&\geq\frac{c}{2}\|x\|_A~~\mbox{for all $n\geq n_0.$}
\end{align}
Since, $S_{\mu_n}$ is the $A$-inverse of $(T-\mu_n I)$ then from Lemma \ref{lemma1-2}(i), for all $x \in \mathcal{H}$ we get
\begin{align*}
	& A(T-\mu_n I)S_{\mu_n}x=AS_{\mu_n}(T-\mu_n I)x=Ax\\
	\implies & (\widetilde{T-\mu_n I})AS_{\mu_n}x=\widetilde{S_{\mu_n}}A(T-\mu_n I)x=Ax\\
	\implies &(\widetilde{T-\mu_n I})\widetilde{S_{\mu_n}}Ax=\widetilde{S_{\mu_n}}(\widetilde{T-\mu_n I})Ax=Ax.
\end{align*}
This implies that $(\widetilde{T}-\mu_n \widetilde{I})\widetilde{S_{\mu_n}}=\widetilde{S_{\mu_n}}(\widetilde{T}-\mu_n \widetilde{I})=\widetilde{I},$ as $R(A)$ is dense in $\textbf{R}(A^{1/2}).$ Thus, for each $n,$ $(\widetilde{T}-\mu_n \widetilde{I})$ is invertible with inverse $\widetilde{S_{\mu_n}}.$ Again, from \eqref{Th1e2} for each $n \geq n_0$ and for all $x \in \mathcal{H},$ we obtain 
$$\|A(T-\mu_n I)x\|_{\textbf{R}(A^{1/2})}\geq\frac{c}{2}\|Ax\|_{\textbf{R}(A^{1/2})}$$
and this implies that
\begin{align}\label{Th1e3}
	\|(\widetilde{T-\mu_n I})Ax\|_{\textbf{R}(A^{1/2})}\geq\frac{c}{2}\|Ax\|_{\textbf{R}(A^{1/2})}.
\end{align}
As $R(A)$ is dense in $\textbf{R}(A^{1/2})$ so \eqref{Th1e3} implies that
\begin{align*}
	\|(\widetilde{T-\mu_n I})A^{1/2}x\|_{\textbf{R}(A^{1/2})}\geq\frac{c}{2}\|A^{1/2}x\|_{\textbf{R}(A^{1/2})}~~\mbox{for all $x \in \mathcal{H}$ and $n \geq n_0$}.
\end{align*}
Hence, for all $n \geq n_0$ we have 
\begin{align*}
	\|\widetilde{S_{\mu_n}}\|_{\textbf{R}(A^{1/2})} \leq \frac{2}{c} \implies 	\|S_{\mu_n}\|_A\leq \frac{2}{c} \implies \|(\mu_n-\lambda)S_{\mu_n}\|_A<1.
\end{align*}
Now, it follows from Lemma \ref{LL2} that $\left(I-(\lambda -\mu_{n_0})S_{\mu_{n_0}}\right)$ is $A$-invertible in $\mathcal{B}_{A^{1/2}}(\mathcal{H}).$ This completes the proof.
\end{proof}

The convexity of $A$-numerical range along with  Proposition \ref{Prop1} and Theorem \ref{Th1} yield the following theorem.

\begin{theorem}\label{T11}
		Let $T \in \mathcal{B}_{A^{1/2}}(\mathcal{H}).$ Then $\sigma_A(T) \subseteq \overline{W_A(T)}.$
\end{theorem}

Next, we present a relation between the $A$-spectrum of $T$ and the spectrum of $\widetilde{T}.$  

\begin{theorem}\label{P2}
	Let $T \in \mathcal{B}_{{A}^{1/2}}(\mathcal{H}).$ Then $\sigma(\widetilde{T}) \subseteq \sigma_A(T).$ Furthermore, if $R(A)$ is closed then $\sigma_A(T)=\sigma(\widetilde{T}).$
\end{theorem}

\begin{proof}
	To prove the first inclusion we suppose $\lambda \in \rho_A(T).$ Then there exists $S \in \mathcal{B}_{{A}^{1/2}}(\mathcal{H})$ such that $A(T-\lambda I)S=AS(T-\lambda I)=A.$ Now, by Lemma \ref{lemma1-2}(i) for all $x \in \mathcal{H},$ we get
	\begin{align}\label{P2E1}
		& A(T-\lambda I)Sx=AS(T-\lambda I)x=Ax\nonumber\\
		\implies &(\widetilde{T-\lambda I})\widetilde{S}Ax=\widetilde{S}(\widetilde{T-\lambda I})Ax=Ax.
	\end{align}
	Since $R(A)$ is dense in $\textbf{R}(A^{1/2})$ so the equality \eqref{P2E1} together with Lemma \ref{lemma1-2}(i) implies that $$(\widetilde{T}-\lambda \widetilde{I})\widetilde{S}=\widetilde{S}(\widetilde{T}-\lambda \widetilde{I})=\widetilde{I},$$ 
	where $\widetilde{I}$ is the identity operator on $\textbf{R}(A^{1/2}).$ Therefore, we have $\lambda \in \rho(\widetilde{T})$ and so $\sigma(\widetilde{T}) \subseteq \sigma_A(T).$
	
	Now, we consider $R(A)$ is closed and take $\lambda \in \rho(\widetilde{T}).$ Then $\widetilde{T}-\lambda \widetilde{I}$ is invertible in $\mathcal{B}(\textbf{R}(A^{1/2}))$ and so there exists $\widetilde{S} \in \mathcal{B}(\textbf{R}(A^{1/2}))$ such that $(\widetilde{T}-\lambda \widetilde{I})\widetilde{S}=\widetilde{S}(\widetilde{T}-\lambda \widetilde{I})=\widetilde{I}.$ Since $R(A)$ is closed then by Lemma \ref{lemma1-2}(ii), there exists unique $T,S \in \mathcal{B}_{{A}^{1/2}}(\mathcal{H})$ such that$\widetilde{T}Ax=ATx$ and $\widetilde{S}Ax=ASx$  for all $x \in \mathcal{H}.$ Therefore, we get
	\begin{align*}
		&(\widetilde{T}-\lambda \widetilde{I})\widetilde{S}Ax=\widetilde{S}(\widetilde{T}-\lambda \widetilde{I})Ax=Ax\\
		\implies & (\widetilde{T}-\lambda \widetilde{I})ASx=\widetilde{S}A(T-\lambda I)x=Ax\\
		\implies & A(T-\lambda I)Sx=AS(T-\lambda I)x=Ax,
	\end{align*}
	for all $x \in \mathcal{H}.$
	Thus, we obtain $A(T-\lambda I)S=AS(T-\lambda I)=A$ and so $\lambda \in \rho_A(T).$ Therefore, we get the desired inclusion $\sigma_A(T) \subseteq \sigma(\widetilde{T}).$ Hence, if $R(A)$ is closed then $\sigma_A(T)=\sigma(\widetilde{T}).$
\end{proof}
We note that  that the closedness of $R(A)$ is not necessary in Theorem \ref {P2}, see Example \ref{ex1}. We also mention that an alternative proof of Theorem \ref {P2} is given in  \cite{baklouti_OAM_2023}. 
Next we prove the spectral mapping theorem for semi-Hilbertian operators, to do so we need  the following  lemma which follows directly from Lemma \ref{lem31}.

\begin{lemma}\label{lem5}
	If $\{T_1,T_2,\ldots,T_n\}\subseteq \mathcal{B}_{A^{1/2}}(\mathcal{H})$ is a set of pairwise commuting elements then $T_1T_2\ldots T_n$ is $A$-invertible in $\mathcal{B}_{A^{1/2}}(\mathcal{H})$ if and only if each $T_i$ is $A$-invertible in $\mathcal{B}_{A^{1/2}}(\mathcal{H}).$
\end{lemma}

\begin{theorem}\label{theo5}
	Let $T \in \mathcal{B}_{A^{1/2}}(\mathcal{H})$ and $p$ be a polynomial with complex coefficients.  Then 
	\[\sigma_{A}(p(T))=p(\sigma_{A}(T))=\{p(\lambda) : \lambda \in \sigma_{A}(T)\}.\]
\end{theorem}

\begin{proof}
	First, we fix a $\lambda \in \mathbb{C}.$ Without loss of generality, we assume that $p(z)=\sum_{i=0}^{n}a_iz^i$ with $a_n \neq 0.$ Now, by the fundamental theorem of algebra we have
	\begin{align}\label{theo5e1}
		p(z)-\lambda=a_n\prod_{i=1}^{n}(z-\lambda_i),
	\end{align}
	where $\lambda_i$ are all the zeroes of the polynomial $p(z)-\lambda.$ Corresponding to \eqref{theo5e1} we have
	\begin{align}\label{theo5e2}
		p(T)-\lambda I=a_n\prod_{i=1}^{n}(T-\lambda_iI).
	\end{align}
	Therefore, by applying Lemma \ref{lem5}  we see that
	\begin{align*}
		\lambda \notin \sigma_{A}(p(T)) &\iff \lambda_i \notin \sigma_{A}(T)~~\mbox{for all $1\leq i \leq n$}\\
		&\iff \lambda \notin p(\sigma_{A}(T))
	\end{align*}
	and the proof is complete.
\end{proof}

In the following theorems we prove some basic results on $A$-normal operators, which justify the new definition of $A$-normal operators introduced in this article.

\begin{theorem}\label{L1}
	Let $T \in \mathcal{B}_{A}(\mathcal{H})$. Then $T$ is $A$-normal if and only if any one of the following holds: \\
	(i)   $\widetilde{T}$ is normal in $\textbf{R}(A^{1/2}).$\\
	(ii)  $T^{\sharp_A}$ is $A$-normal.\\
	(iii) $\|Tx\|_A=\|T^{\sharp_A}x\|_A$ for all $x \in \mathcal{H}.$
\end{theorem}

\begin{proof}
 (i) Suppose that $T$ is $A$-normal. Then by Lemma \ref{lemma1-2} (i) and Lemma \ref{lemma3-6} (i), for all  $x \in \mathcal{H}$ we get
	\begin{align*}
		AT^{\sharp_A}Tx=ATT^{\sharp_A}x &\implies \widetilde{T^{\sharp_A}T}Ax=\widetilde{TT^{\sharp_A}}Ax\\
		& \implies (\widetilde{T})^*\widetilde{T}Ax=\widetilde{T}(\widetilde{T})^*Ax.
	\end{align*}
	Therefore, we obtain
	\begin{align*}
		Z_AT^{\sharp_A}T=\widetilde{T^{\sharp_A}T}Z_A=(\widetilde{T})^*\widetilde{T}Z_A=\widetilde{T}(\widetilde{T})^*Z_A.
	\end{align*}
	Now, by Lemma \ref{lemma1-2} (i) we have $(\widetilde{T})^*\widetilde{T}=\widetilde{T}(\widetilde{T})^*.$ 
	
	To prove the converse part we consider $\widetilde{T}$ is normal in $\textbf{R}(A^{1/2})$. Then by applying Lemma \ref{lemma1-2} (i) and Lemma \ref{lemma3-6} (i), we have
	\begin{align*}
		(\widetilde{T})^*\widetilde{T}=\widetilde{T}(\widetilde{T})^*
		\implies &\widetilde{T^{\sharp_A}T}=\widetilde{TT^{\sharp_A}}\\
		\implies & \widetilde{T^{\sharp_A}T}Ax=\widetilde{TT^{\sharp_A}}Ax \,\, \mbox{for all $x \in \mathcal{H}$}\\
		\implies & AT^{\sharp_A}Tx=ATT^{\sharp_A}x \,\, \mbox{for all $x \in \mathcal{H}$}\\
		\implies & AT^{\sharp_A}T=ATT^{\sharp_A},
	\end{align*}
	as desired.

		(ii)	We have,  
	\begin{align}\label{n1}
	\mbox{$T$ is $A$-normal}&\iff 	ATT^{\sharp_A}=AT^{\sharp_A}T\nonumber \\
	&\iff (T^{\sharp_A})^*T^*A=T^*(T^{\sharp_A})^*A\nonumber\\
		& \iff (T^{\sharp_A})^*AT^{\sharp_A}=T^*AT^{\sharp_A \sharp_A}\nonumber\\
		& \iff AT^{\sharp_A \sharp_A}T^{\sharp_A}=AT^{\sharp_A}T^{\sharp_A \sharp_A},
	\end{align}
	as required.
	
	(iii) For all $x \in \mathcal{H},$ we have
	\begin{align*}
		AT^{\sharp_A}T=ATT^{\sharp_A} \iff & \langle AT^{\sharp_A}Tx,x \rangle = \langle ATT^{\sharp_A}x,x \rangle\\
		\iff &  \langle T^*ATx,x \rangle = \langle TT^{\sharp_A}x,x \rangle_A\\
		\iff &  \langle Tx,Tx \rangle_A = \langle T^{\sharp_A}x,T^{\sharp_A}x \rangle_A\\
		\iff & \|Tx\|_A=\|T^{\sharp_A}x\|_A,
	\end{align*}
	as desired.

\end{proof}

\begin{theorem}\label{t2}
	Let $T \in \mathcal{B}_{A}(\mathcal{H})$.   If $T$ is $A$-normal  then \\
		(i)  $(T-\lambda I)$ is $A$-normal, for all $\lambda \in \mathbb{C}.$ \\
		(ii)   $r_A(T)=\|T\|_A.$\\
\end{theorem}
\begin{proof}
(i) As $T$ is $A$-normal then it follows from Theorem \ref{L1}  (i) that $\widetilde{T}$ is normal. Hence, for all $\lambda \in \mathbb{C},$ $\widetilde{T}-\lambda \widetilde{I}$ is normal and so $\widetilde{T-\lambda I}$ is also normal. Again it follows from Theorem \ref{L1} (i) that $(T-\lambda I)$ is $A$-normal.

(ii) Since $T$ is $A$-normal so by Theorem \ref{L1} (i), $\widetilde{T}$ is normal. Then from \eqref{r1}, we have
\begin{align*}
	r_{A}(T)=\lim_{n \to \infty}\|T^n\|^{1/n}_{A}
	&=\lim_{n \to \infty}\|\widetilde{T^n}\|^{1/n}_{\mathcal{B}(\textbf{R}(A^{1/2}))}  ~~\mbox{\big(by Lemma \ref{lemma3-6} (ii)\big)}\\
	&=\lim_{n \to \infty}\|\widetilde{T}^n\|^{1/n}_{\mathcal{B}(\textbf{R}(A^{1/2}))}~~\mbox{\big(by Lemma \ref{lemma1-2} (i)\big)}\\
	&= r(\widetilde{T})=\|\widetilde{T}\|_{\mathcal{B}(\textbf{R}(A^{1/2}))}=\|T\|_A.
\end{align*}
This completes the proof.

\end{proof}


	



To prove our next result we introduce the following definition.

\begin{definition}
	An operator $T \in \mathcal{B}_{A}(\mathcal{H})$ is said to be $A$-invertible in $\mathcal{B}_{A}(\mathcal{H})$ if there exists $S \in \mathcal{B}_{A}(\mathcal{H})$ such that 
	$ATS=AST=A$ and $S$ is said to be an $A$-inverse of $T$ in $\mathcal{B}_{A}(\mathcal{H}).$
\end{definition}

\begin{lemma}
	Let $T \in \mathcal{B}_{A}(\mathcal{H})$ be $A$-invertible in $\mathcal{B}_{A}(\mathcal{H})$ with $A$-inverse $S.$ Then $T^{\sharp_A}$ is $A$-invertible in $\mathcal{B}_{A}(\mathcal{H})$ with $A$-inverse $S^{\sharp_A}.$
\end{lemma}

\begin{proof}
	Since $T \in \mathcal{B}_{A}(\mathcal{H})$ is $A$-invertible in $\mathcal{B}_{A}(\mathcal{H})$ with an $A$-inverse $S$ then for all $x \in \mathcal{H},$ we obtain
	\begin{align*}
	&	S^*T^*Ax=T^*S^*Ax=Ax\\
	\implies & S^*AT^{\sharp_A}x=T^*AS^{\sharp_A}x=Ax\\
	\implies & AS^{\sharp_A}T^{\sharp_A}x=AT^{\sharp_A}S^{\sharp_A}x=Ax
	\end{align*}
	and hence we have 
	\begin{align}\label{n}
	AS^{\sharp_A}T^{\sharp_A}=AT^{\sharp_A}S^{\sharp_A}=A,
\end{align}
	 as desired.
\end{proof}

\begin{proposition}\label{prop2}
	Let $T \in \mathcal{B}_{A}(\mathcal{H})$ be $A$-invertible in $ \mathcal{B}_{A}(\mathcal{H})$ with an $A$-inverse $S.$ If $T$ is $A$-normal then $S$ is $A$-normal.
\end{proposition}

\begin{proof}
	Since $T$ is $A$-normal then from \eqref{n1} using the relations $R(T^{\sharp_A})\subseteq \overline{R(A)}$ and $A^{\dagger}A=P_{\overline{R(A)}},$ we get
	\begin{align}\label{n2}
		T^{{\sharp_A}{\sharp_A}}T^{\sharp_A}=T^{\sharp_A}T^{{\sharp_A}{\sharp_A}}
	\end{align}
    Again from the equality \eqref{n}, we have
    \begin{align}\label{n3}
    	S^{\sharp_A}T^{\sharp_A}=T^{\sharp_A}S^{\sharp_A}=P_{\overline{R(A)}}.
    \end{align}
	Now, by using the Moore-Penrose equation (see \cite{Penr55}), we
	observe that
	\begin{eqnarray*}
		T^{\sharp_A}P_{\overline{R(A)}}=P_{\overline{R(A)}}T^{\sharp_A}=T^{\sharp_A}.
	\end{eqnarray*}
 Therefore, by applying \eqref{n2} and \eqref{n3}, we obtain
	\begin{align*}
		S^{\sharp_A}S^{{\sharp_A}{\sharp_A}}=S^{\sharp_A}S^{{\sharp_A}{\sharp_A}}P_{\overline{R(A)}}
		&=S^{\sharp_A}S^{{\sharp_A}{\sharp_A}}T^{\sharp_A}T^{{\sharp_A}{\sharp_A}}S^{{\sharp_A}{\sharp_A}}S^{\sharp_A}\\
		&=S^{\sharp_A}S^{{\sharp_A}{\sharp_A}}T^{{\sharp_A}{\sharp_A}}T^{\sharp_A}S^{{\sharp_A}{\sharp_A}}S^{\sharp_A}\\
		&=P_{\overline{R(A)}}S^{{\sharp_A}{\sharp_A}}S^{\sharp_A}\\
		&=S^{{\sharp_A}{\sharp_A}}S^{\sharp_A}.
	\end{align*}
	Thus, we have $S^{\sharp_A}$ is $A$-normal and Theorem \ref{L1} (ii)
	implies that $S$ is $A$-normal.
\end{proof}

%


Now, we prove an important relation concerning the $A$-numerical range and $A$-spectrum for  $A$-normal operators.

\begin{theorem}\label{T12}
	If $T$ is $A$-normal then $\overline{W_A(T)}=conv(\sigma_A(T)).$
\end{theorem}

\begin{proof}
	Since $T$ is $A$-normal then by Theorem \ref{L1} (i), $\widetilde{T}$ is normal in $\textbf{R}(A^{1/2}).$ Now, by applying \cite[Theorem 1.4-4]{Gustafson_book} for the normal operator $\widetilde{T},$ we obtain  $W(\widetilde{T}) \subseteq conv(\sigma(\widetilde{T})).$ This relation together with Lemma \ref{lemma3-6} (iv) and Theorem \ref{P2} implies that 
	\begin{align}\label{Te1}
	&	W_A(T) \subseteq conv(\sigma_A(T)) \nonumber \\
		\implies  & \overline{W_A(T)} \subseteq \overline{conv(\sigma_A(T))}. 
	\end{align}
	On the other hand from Theorem \ref{T11} we get 
		\begin{align}\label{Te2}
		&\sigma_A(T) \subseteq \overline{W_A(T)} \nonumber \\
		\implies &	conv(\sigma_A(T)) \subseteq conv(\overline{W_A(T)}).
	\end{align}
	Observe that closure of the convex set $W_A(T)$ is convex and convex hull of the closed set $\sigma_A(T)$ is closed and so from \eqref{Te1} and \eqref{Te2} it follows that 
	\[  conv(\overline{W_A(T)}) = \overline{W_A(T)}  \subseteq \overline{conv(\sigma_A(T))}= conv(\sigma_A(T)) \subseteq conv(\overline{W_A(T)}) .\]
	This completes the proof. 
\end{proof}

Here we note that for $A=I$, Theorem \ref{T12} reduces to the classical result that for a normal operator on a Hilbert space the closure of the numerical range is the convex hull of the spectrum (see  \cite{Berberian_DJM_1981}).

\begin{remark}
  It can be easily proved that Theorem \ref{T12} also holds for $A$-hyponormal operators. Observe that an operator $T \in\mathcal{B}_{A}(\mathcal{H})$ is said to be $A$-hyponormal if $AT^{\sharp_A}T\geq ATT^{\sharp_A}.$ 
\end{remark}

\begin{cor}\label{cor1}
	If $T$ is $A$-normal and $rank(A)< \infty,$ then $W_A(T)=conv(\sigma_{A_p}(T)).$
\end{cor}
\begin{proof}
	Since $rank(A)< \infty,$ so $R(A)$ is closed and $\sigma(\widetilde{T})$ is a finite set and the elements are precisely the point spectrums of $\widetilde{T}$. Thus from Theorem \ref{P2} and Proposition \ref{P3}, we get $\sigma_{A}(T)=\sigma_p(\widetilde{T})=\sigma_{A_p}(T).$ Again as  $rank(A)< \infty,$ so $W(\widetilde{T})$ is compact and by applying Lemma \ref{lemma3-6} (iv) we obtain $W_A(T)$ is also compact. Now, the desired result follows immediately from Theorem \ref{T12}. 
\end{proof}
\begin{example}
	If we consider the matrices $\mathbb{A}=\begin{pmatrix}
		A_1&0\\
		0&A_2
	\end{pmatrix}$ and $\mathbb{T}=\begin{pmatrix}
	T_1&0\\
	0&T_2
\end{pmatrix},$ where ${A_1}=\begin{pmatrix}
1&0\\
0&1
\end{pmatrix},$ $A_2=\begin{pmatrix}
2&0&0\\
0&1&0\\
0&0&1
\end{pmatrix},$ ${T_1}=\begin{pmatrix}
0&1\\
0&0
\end{pmatrix}$ and $T_2=\begin{pmatrix}
2i&0&0\\
0&-\frac32-i&0\\
0&0&\frac32-i
\end{pmatrix}.$ Then it is easy to verify that $W_{\mathbb{A}}(\mathbb{T})=conv\{2i,-\frac32-i,\frac32-i\}$ but $\mathbb{T}$ is not  $\mathbb{A}$-normal. Therefore, the converse of Corollary \ref{cor1} is not true in general.
\end{example}

The final result (justifying the introduction of new definition of $A$-normal operator) in this section follows easily from Theorem \ref{T11} and Theorem \ref{t2} (ii).
\begin{cor}
		If $T$ is $A$-normal then $r_A(T)=w_A(T)=\|T\|_A.$
\end{cor}

\section{Anderson's theorem on semi-Hilbertian operators}\label{s2}

In 1970, Anderson in an unpublished note proved that if $T$ is a linear operator on $ \mathbb{C}^n$ such that the numerical range $W(T)$  is contained in the closed unit disk  $\overline{\mathbb{D}}$ and  it intersects the boundary $\partial\mathbb{D}$  of the disk at more than $n$ points, then $W(T)=\overline{\mathbb{D}}.$ This remarkable result is known as the Anderson's theorem and the proof is available in \cite{DWMAMS,TYLAA}. The infinite dimensional analogue of Anderson's theorem is studied in  \cite{Gau_PAMS}, in the form that,   if $T$ is a compact operator on a Hilbert space such that  $W(T)$ is contained in $\overline{ \mathbb{D}}$ and $\overline{W(T)}$ intersects  $\partial \mathbb{D}$ at infinitely many points, then $W(T)=\overline{ \mathbb{D}}.$  A nice generalization of the result  is studied in 
\cite{BSS_JMA_2018}, it says that if the compact operator is perturbed by a normal operator with the additional condition that $\sigma_{ess}(T) \subseteq {\mathbb{D}} $ then  the result holds true.  To be precise, the statement goes as : 
\begin{theorem}\cite[Th. 4]{BSS_JMA_2018}\label{l1}
	Let $T=N+K,$ where $N$ is normal and $K$ is compact operator on a Hilbert space $\mathcal{H}.$ If $W(T) \subseteq \overline{\mathbb{D}}$, $\sigma_{ess}(T) \subseteq \mathbb{D}$ and $\overline{W(T)}\cap \partial \mathbb{D}$ is an infinite set, then $W(T)=\overline{\mathbb{D}}.$
\end{theorem}

Recently, Anderson's theorem and its compact extension has been studied in the setting of semi-Hilbertian operators, see \cite{BKPS_LAA_23}.  
Our main goal of this section is to  generalize \cite[Th. 4]{BSS_JMA_2018} in the semi-Hilbertian structure. For this we need to introduce the following definitions in the setting of semi-Hilbertian operators in such a way that for $A=I$ it reduces to the Atkinson's characterization of Fredholm operators \cite{Atkinson} and essential spectrum \cite{FSW_1972}, respectively.

\begin{definition}
	An operator $T \in \mathcal{B}_{A^{1/2}}(\mathcal{H})$ is said to be $A$-Fredholm if there exists $S \in \mathcal{B}_{A^{1/2}}(\mathcal{H})$ and $K_1,K_2 \in \mathcal{K}_{{A}^{1/2}}(\mathcal{H})$ such that $ATS-A=AK_1$ and $AST-A=AK_2.$
\end{definition}

\begin{definition}
	For $T \in \mathcal{B}_{A^{1/2}}(\mathcal{H}),$ the $A$-essential spectrum of $T,$  denoted by $\sigma_{{A}_{ess}}(T),$  is defined as
	$$\sigma_{{A}_{ess}}(T)=\{\lambda \in \mathbb{C} : (T-\lambda I)~~ \mbox{is not $A$-Fredholm}\}.$$
\end{definition}
 Observe  that if $\mathcal{H}$ is an infinite dimensional Hilbert space then $T \in \mathcal{K}(\mathcal{H})$ if and only if $\sigma_{ess}(T)=\{0\},$ see \cite[p. 9]{wu_gau_book}.
We next study the relation between $A$-essential spectrum and $A$-spectrum. 
\begin{proposition}
	If $T \in \mathcal{B}_{A^{1/2}}(\mathcal{H})$ then $\sigma_{{A}_{ess}}(T) \subseteq \sigma_A(T).$
\end{proposition}

\begin{proof}
	Suppose that $\lambda \notin \sigma_A(T).$ Then there exists $S \in \mathcal{B}_{A^{1/2}}(\mathcal{H})$ such that $A(T-\lambda I)S-A=AS(T-\lambda I)-A=0.$ As $0 \in \mathcal{K}_{{A}^{1/2}}(\mathcal{H})$ hence $\lambda \notin \sigma_{{A}_{ess}}(T),$ as desired.  
\end{proof}

In the next theorem we prove that $A$-essential spectrum of $T$ and  essential spectrum of $T$ are equal, if $R(A)$ is closed.

\begin{theorem}\label{p1}
	Let $T \in \mathcal{B}_{A^{1/2}}(\mathcal{H})$ and let $R(A)$ be closed. Then $\sigma_{{A}_{ess}}(T) =\sigma_{ess}(\widetilde{T}).$
\end{theorem}

\begin{proof}
	Suppose that $\lambda \notin \sigma_{{A}_{ess}}(T).$ Then there exists $S \in \mathcal{B}_{A^{1/2}}(\mathcal{H})$ and $K_1,K_2 \in \mathcal{K}_{{A}^{1/2}}(\mathcal{H})$ such that $A(T-\lambda I)S-A=AK_1$ and $AS(T-\lambda I)-A=AK_2.$ Now, by applying Lemma \ref{lemma1-2}(i) for all $x \in \mathcal{H},$ we get 
	$(\widetilde{(T-\lambda I)}\widetilde{S}-\widetilde{I})Ax=\widetilde{K_1}Ax$ and this implies that $\widetilde{(T-\lambda I)}\widetilde{S}-\widetilde{I}=\widetilde{K_1},$ as $R(A)$ is dense in $\textbf{R}(A^{1/2}).$
	Since $R(A)$ is closed so it follows from Lemma \ref{lemma3-6}(iii) that $\widetilde{K_1}$ is compact. Similarly, we get $\widetilde{S}\widetilde{(T-\lambda I)}-\widetilde{I}=\widetilde{K_2},$ where $\widetilde{K_2}$ is also compact. Thus $\lambda \notin \sigma_{ess}(\widetilde{T}).$ 
	
	Conversely, suppose that $\lambda \notin \sigma_{ess}(\widetilde{T}).$ Then there exists $\widetilde{S} \in \mathcal{B}(\textbf{R}(A^{1/2}))$ and $\widetilde{K_1},\widetilde{K_2} \in \mathcal{K}(\textbf{R}(A^{1/2}))$ such that
	\begin{align}\label{p1e1}
	 \widetilde{(T-\lambda I)}\widetilde{S}-\widetilde{I}=\widetilde{K_1}
	 \end{align} 
     and 
     \begin{align}\label{p1e2}
     \widetilde{S}\widetilde{(T-\lambda I)}-\widetilde{I}=\widetilde{K_2}.
	\end{align}
	  Since $R(A)$ is closed then by Lemma \ref{lemma1-2}(ii) and Lemma \ref{lemma3-6}(iii) there exists unique $S \in \mathcal{B}_{A^{1/2}}(\mathcal{H})$ and $K_1,K_2 \in \mathcal{K}_{{A}^{1/2}}(\mathcal{H})$ such that $ASx=\widetilde{S}Ax$ and $AK_ix=\widetilde{K_i}Ax$ for $i=1,2$ and $x \in \mathcal{H}.$ Now, from the equalities \eqref{p1e1} and \eqref{p1e2} we have $A(T-\lambda I)S-A=AK_1$ and $AS(T-\lambda I)-A=AK_2.$ Hence, $\lambda \notin \sigma_{{A}_{ess}}(T)$ and this completes the proof.	  
\end{proof}

If we consider the operators mentioned in Example \ref{ex1} then we obtain that $\sigma_{{A}_{ess}}(T)=\sigma_{ess}(\widetilde{T})=\{\lambda_0\}.$ Thus the closedness of $R(A)$ is not necessary in Theorem \ref{p1}.

The  next corollary follows from Theorem \ref{p1} and Lemma \ref{lemma3-6} (iii).

\begin{cor}\label{c1}
	Let $T \in \mathcal{B}_{A^{1/2}}(\mathcal{H})$.
	     If $rank(A) < \infty$ then  $\sigma_{{A}_{ess}}(T)=\phi.$ 
		 If $ rank(A)=\infty$ and $R(A)$ is closed then $T \in \mathcal{K}_{A^{1/2}}(\mathcal{H}) $ if and only if $ \sigma_{{A}_{ess}}(T)=\{0\}.$
\end{cor}

We are now in a position to prove the desired result.

\begin{theorem}\label{t1}
	Let $T \in \mathcal{B}_{A^{1/2}}(\mathcal{H})$ be such that $T=N+K,$ where $N$ is $A$-normal and $K \in \mathcal{K}_{{A}^{1/2}}(\mathcal{H}).$ Let $R(A)$ be closed. If $W_A(T) \subseteq \overline{\mathbb{D}}$, $\sigma_{{A}_{ess}}(T) \subseteq \mathbb{D}$ and $\overline{W_A(T)}\cap \partial \mathbb{D}$ is an infinite set, then $W_A(T)=\overline{\mathbb{D}}.$
\end{theorem}

\begin{proof}
	Since,  $T \in \mathcal{B}_{A^{1/2}}(\mathcal{H})$ so by Lemma \ref{lemma1-2} (i), we get $\widetilde{T}=\widetilde{N}+\widetilde{K}.$ As $R(A)$ is closed thus from Theorem \ref{L1} (i) and Lemma \ref{lemma3-6} (iii) it follows that $\widetilde{N}$ and $\widetilde{K}$ are normal and compact operators in $\textbf{R}(A^{1/2}),$ respectively. Now, by applying Lemma \ref{l1} for the operator $\widetilde{T}$ on the Hilbert space $\textbf{R}(A^{1/2}),$ we obtain that if $W(\widetilde{T}) \subseteq \overline{\mathbb{D}}$, $\sigma_{ess}(\widetilde{T}) \subseteq \mathbb{D}$ and $\overline{W(\widetilde{T})}\cap \partial \mathbb{D}$ is an infinite set, then $W(\widetilde{T})=\overline{\mathbb{D}}.$ Finally, the desired result follows from Theorem \ref{p1} and Lemma \ref{lemma3-6} (iv).
\end{proof}

Considering $N=0$ in Theorem \ref{t1} and by applying Corollary  \ref{c1}, we get the next result which was recently proved in \cite[Th. 2.7]{BKPS_LAA_23}.

 \begin{cor}\label{c2}
	Let $T\in \mathcal{K}_{{A}^{1/2}}(\mathcal{H})$ and $R(A)$ be closed. If ${W_A(T)} \subseteq \overline{\mathbb{D}}$ and $\overline{W_A(T)}$ intersects $\partial \mathbb{D}$ at infinitely many points, then $W_A(T) = \overline{\mathbb{D}}.$
\end{cor}
Clearly, for $A=I$ Theorem \ref{t1} reduces to existing result \cite[Th. 4]{BSS_JMA_2018}.

\section{strongly $A$-numerically closed classes of $\mathcal{B}_{{A}^{1/2}}(\mathcal{H})$}\label{s3}

 An operator class $ \mathcal{C} \subseteq \mathcal{B}(\mathcal{H})$ is said to be strongly numerically closed, if for any $T \in \mathcal{C}$ and any $ \epsilon >0$, there exists a compact operator $K$ such that $\|K\| < \epsilon,$  $ T + K \in \mathcal{C}$ and $W(T + K) $ is closed.  The study of strongly numerically closed operator class was motivated by Bourin \cite{Bourin_JOT_2003} and then followed in \cite{Ji_oam_18,zhu_BJMA_2015}.  Analogously, we define an operator class $\mathcal{C} \subseteq \mathcal{B}_{{A}^{1/2}}(\mathcal{H}) $  to be strongly $A$-numerically closed if for any $T \in \mathcal{C}$ and any $\epsilon>0$ there exists $K \in \mathcal{K}_{{A}^{1/2}}(\mathcal{H})$ with $\|K\|_A<\epsilon$ such that $T+K \in \mathcal{C}$ and $\overline{W_A(T+K)}=W_A(T+K).$  It is known that the class of normal and hyponormal operators are all strongly numerically closed. Recently \cite{liang_JMAA_2023}, it is proved that the class of complex symmetric operators defined on Hilbert spaces are strongly numerically closed.  Our plan in this section is to study the strongly $A$-numerically closed class of operators. Following the work done in \cite{Ji_oam_18,zhu_BJMA_2015}, we can easily show that $A$-normal and $A$-hyponormal operators are all strongly $A$-numericallly closed class of operators.  Motivated by the notion of complex symmetric operators studied in \cite{liang_JMAA_2023} we here introduce the following definition.

\begin{definition}
	An operator $T \in \mathcal{B}_{{A}^{1/2}}(\mathcal{H})$ is said to induce a complex symmetric operator (induce CSO)  if $\widetilde{T} \in \mathcal{B}(\textbf{R}(A^{1/2}))$ is a complex symmetric operator. 
\end{definition}
Since normal operators are complex symmetric \cite{Garcia_TAMS_2006}, then from Theorem \ref{L1} (i) it follows that if $T$ is $A$-normal then $T$ induces $CSO.$

\begin{proposition}\label{P1}
	Let $T \in \mathcal{B}_{A^{1/2}}(\mathcal{H})$ and $R(A)$ be closed. Then $T$ induces $CSO$ if and only if there exists an orthonormal basis $\{Ax_n : n\in \mathbb{N}\}$ of $\textbf{R}(A^{1/2})$ such that $\langle Tx_n, x_m\rangle_A =\langle Tx_m, x_n \rangle_A.$
\end{proposition}

\begin{proof}
	Suppose $T$ induces $CSO$. Then $\widetilde{T}$ is a $CSO$ on the Hilbert space $\textbf{R}(A^{1/2}).$ Since $R(A)$ is closed so $R(A^{1/2})=R(A)$ and by \cite[Prop. 2]{Garcia_TAMS_2006} there exists an orthonormal basis $\{Ax_n : n\in \mathbb{N}\}$ of $R(A)$ such that
	\begin{align*}
		&\big( \widetilde{T} Ax_n, Ax_m \big)_{\textbf{R}(A^{1/2})}=\big( \widetilde{T} Ax_m, Ax_n \big)_{\textbf{R}(A^{1/2})}\\
		\implies & \big( ATx_n, Ax_m \big)_{\textbf{R}(A^{1/2})}=\big( ATx_m, Ax_n \big)_{\textbf{R}(A^{1/2})}\,\,\,\,\big(\mbox{by Lemma \ref{lemma1-2}(i)}\big)\\
		\implies & \langle Tx_n, x_m\rangle_A =\langle Tx_m, x_n \rangle_A,
	\end{align*}
	as desired.
	
	Conversely, suppose that there exists an orthonormal basis $\{Ax_n : n\in \mathbb{N}\}$ of $\textbf{R}(A^{1/2})$ such that $\langle Tx_n, x_m\rangle_A =\langle Tx_m, x_n \rangle_A.$ Then by Lemma \ref{lemma1-2} (i) we have
	\begin{align*}
		&\langle Tx_n, x_m\rangle_A =\langle Tx_m, x_n \rangle_A\\
		\implies& \big( ATx_n, Ax_m \big)_{\textbf{R}(A^{1/2})}=\big( ATx_m, Ax_n \big)_{\textbf{R}(A^{1/2})}\\
		\implies &\big( \widetilde{T} Ax_n, Ax_m \big)_{\textbf{R}(A^{1/2})}=\big( \widetilde{T} Ax_m, Ax_n \big)_{\textbf{R}(A^{1/2})}.
	\end{align*}
	Hence by  \cite[Prop. 2]{Garcia_TAMS_2006} we conclude that $\widetilde{T}$ is a $CSO.$ This completes the proof.
\end{proof}

	Now we give an example of an operator $T$ which is not $CSO$ but it induces $CSO.$
\begin{example}
 The operator $T : \mathbb{C}^3 \to \mathbb{C}^3$ defined by the matrix $ T= \begin{pmatrix}
		1&a&0\\
		0&0&b\\
		0&0&1
	\end{pmatrix}$ is $CSO$ if and only if $|a|=|b|,$ see \cite[Example 7]{Garcia_TAMS_2006}. So $ T= \begin{pmatrix}
		1&1&0\\
		0&0&0\\
		0&0&1
	\end{pmatrix}$ is in $\mathcal{B}_{A^{1/2}}(\mathcal{H})$ but not a $CSO.$ Considering $ A= \begin{pmatrix}
		1&0&0\\
		0&0&0\\
		0&0&0
	\end{pmatrix},$ we have $R(A)=span\{Ae_1\},$ where $ e_1= \begin{pmatrix}
		1\\
		0\\
		0
	\end{pmatrix}.$ Now, it follows from Proposition \ref{P1} that $T$ induces $CSO.$
\end{example}
Next, we prove the following proposition. 
\begin{proposition}\label{pp1}
	Let $T \in \mathcal{B}_{A^{1/2}}(\mathcal{H})$ and $R(A)$ be closed. If rank of $AT$ is finite then $T \in \mathcal{K}_{A^{1/2}}(\mathcal{H}).$
\end{proposition}

\begin{proof}
	Since $R(A)$ is closed and $rank(AT)<\infty$ so $\widetilde{T}$ is a finite rank operator on the Hilbert space $\textbf{R}(A^{1/2}).$ Thus $\widetilde{T}$ is a compact operator and it follows from Lemma \ref{lemma3-6} (iii) that $T \in \mathcal{K}_{A^{1/2}}(\mathcal{H}).$
\end{proof}

However, the converse is not necessarily true. If we consider $A=I$ and 
$T : \ell_2 \to \ell_2$ defined as $T(x_1,x_2,\ldots)=(\frac{x_1}{1},\frac{x_2}{2},\ldots)$ then $T\in \mathcal{K}_{A^{1/2}}(\mathcal{H})$ but rank of $AT$ is not finite.

We next prove  the following theorem in the setting of semi-Hilbertian structure.

\begin{theorem}\cite{Bourin_JOT_2003}\label{lemma7}
	Let $T \in \mathcal{B}(\mathcal{H}).$ Then for $\epsilon>0$ there exists normal, finite rank operator $N$ with $\|N\|<\epsilon$ such that $\overline{W(T+N)}=W(T+N).$
\end{theorem}


\begin{theorem}\label{T1}
	Let $T \in \mathcal{B}_{{A}^{1/2}}(\mathcal{H})$ and $R(A)$ be closed. Then given any $\epsilon>0$ there exists a $A$-normal operator $N$ with $\|N\|_A<\epsilon$ and $AN$ is of finite rank such that $\overline{W_A(T+N)}=W_A(T+N).$
\end{theorem}

\begin{proof}
	Since  $T \in \mathcal{B}_{A^{1/2}}(\mathcal{H})$, then by Theorem \ref{lemma1-2} (i) there exists a unique $\widetilde{T} \in
	\mathcal{B}(\textbf{R}(A^{1/2}))$ such that $Z_AT=\widetilde{T} Z_A$, where $Z_Ax=Ax$ for all $x\in \mathcal{H}.$ Now given  $\epsilon>0$ and  $\widetilde{T} \in \mathcal{B}(\textbf{R}(A^{1/2})),$  applying Lemma \ref{lemma7}  we get a  normal, finite rank operator $\widetilde{N} \in \mathcal{B}(\textbf{R}(A^{1/2}))$ with $\|\widetilde{N}\|_{\mathcal{B}(\textbf{R}(A^{1/2}))}<\epsilon$ such that $\overline{W(\widetilde{T}+\widetilde{N})}=W(\widetilde{T}+\widetilde{N}).$ As $R(A)$ is closed so it follows from  Lemma \ref{lemma1-2} (ii) that there exists a unique $N \in \mathcal{B}_{A}(\mathcal{H})$ satisfying $Z_AN=\widetilde{N}Z_A.$ Therefore, by applying Theorem \ref{L1} (i) we get $N$ is $A$-normal. Now, it follows from Lemma \ref{lemma1-2} (i) and Lemma \ref{lemma3-6} (iv) that $\overline{W_A(T+N)}=W_A(T+N).$ Also by applying Lemma \ref{lemma3-6} (ii), we get $\|N\|_A<\epsilon.$ As $R(A)$ is closed so $R(A)=R(A^{1/2})$ and by Lemma \ref{lemma1-2} (i) it follows that $\widetilde{N}(R(A))=AN(\mathcal{H}).$ Since $\widetilde{N}$ is of finite rank so rank of $AN$ is also of finite rank. This completes the proof.
\end{proof}

Recently in \cite{liang_JMAA_2023} it is proved that CSO class is strongly numerically closed class.

\begin{theorem}\cite{liang_JMAA_2023}\label{lemma8}
	Let $T$ be a $CSO.$ Then given any $\epsilon>0,$ there exists a $K \in \mathcal{K}(\mathcal{H})$ with $\|K\|<\epsilon,$ such that $T+K \in CSO$ and $\overline{W(T+K)}=W(T+K).$
\end{theorem}

We here generalize the same in the setting of semi-Hilbertian structure. 
\begin{theorem}\label{T2}
	Let $T \in \mathcal{B}_{A^{1/2}}(\mathcal{H})$ and $R(A)$ be closed. If $T$ induces $CSO$ then given $\epsilon>0$ there exists $K \in \mathcal{K}_{{A}^{1/2}}(\mathcal{H})$ with $\|K\|_A<\epsilon,$ such that $T+K$ induces $CSO$ and $\overline{W_A(T+K)}=W_A(T+K).$
\end{theorem}

\begin{proof}
	Since  $T$ induces $CSO$ then from the definition of  $\widetilde{T} $ it follows that  $\widetilde{T} $  is a $CSO$ in $ \mathcal{B}(\textbf{R}(A^{1/2})).$  For the CSO operator $\widetilde{T} $ and $\epsilon>0,$ using Theorem \ref{lemma8},  we get $\widetilde{K}\in \mathcal{K}(\textbf{R}(A^{1/2}))$ with $\|\widetilde{K}\|_{\mathcal{B}(\textbf{R}(A^{1/2}))}<\epsilon$ such that $\widetilde{T}+\widetilde{K} \in CSO$ and $\overline{W(\widetilde{T}+\widetilde{K})}=W(\widetilde{T}+\widetilde{K}).$ Since $R(A)$ is closed so it follows from  Lemma \ref{lemma1-2} (ii) that there exists a unique $K \in \mathcal{B}_{A^{1/2}}(\mathcal{H})$  such that $Z_AK= \widetilde{K}Z_A.$  Now from Lemma \ref{lemma3-6} (iii) we get $K \in \mathcal{K}_{A^{1/2}}(\mathcal{H}).$ Since $\widetilde{T+K}=\widetilde{T}+\widetilde{K} \in$ $CSO$ so from Lemma \ref{lemma1-2} (ii) it follows that $T+K$ induces $CSO.$  Now, from Lemma \ref{lemma1-2} (i) and Lemma \ref{lemma3-6} (iv) we have, $\overline{W_A(T+K)}=W_A(T+K).$ Finally, by applying Lemma \ref{lemma3-6} (ii) we get $\|K\|_A<\epsilon,$ and this completes the proof.
\end{proof}

Finally we show that the class of $A$-normal operators is strongly $A$-numerically closed. For this we need the following result.
\begin{theorem}\cite{zhu_BJMA_2015}\label{theorem9}
	If $T$ is normal and if $\epsilon>0,$ then there exists a $K \in \mathcal{K}(\mathcal{H})$  with $\|K\|<\epsilon,$ such that $T+K$ is normal and $\overline{W(T+K)}=W(T+K).$
\end{theorem}

\begin{theorem}\label{T3}
	Let $T$ be $A$-normal and $R(A)$ be closed. Then given $\epsilon>0$ there exists $K \in \mathcal{K}_{{A}^{1/2}}(\mathcal{H})$ with $\|K\|_A<\epsilon,$ such that $T+K$ is $A$-normal and $\overline{W_A(T+K)}=W_A(T+K).$
\end{theorem}

\begin{proof}
	Since  $T$ is $A$-normal then from Theorem \ref{L1} (i) we have $\widetilde{T}$ is normal. Now given  $\epsilon>0$ and   $\widetilde{T} \in  {\mathcal{B}(\textbf{R}(A^{1/2}))} $ using Theorem  \ref{theorem9},  we get  $\widetilde{K}\in \mathcal{K}(\textbf{R}(A^{1/2}))$ with $\|\widetilde{K}\|_{\mathcal{B}(\textbf{R}(A^{1/2}))}<\epsilon$ such that $\widetilde{T}+\widetilde{K} $ is normal and $\overline{W(\widetilde{T}+\widetilde{K})}=W(\widetilde{T}+\widetilde{K}).$ As $R(A)$ is closed so from  Lemma \ref{lemma1-2} (ii) there exists a unique $K \in \mathcal{B}_{A^{1/2}}(\mathcal{H})$  such that $Z_AK= \widetilde{K}Z_A.$ Now from Lemma \ref{lemma3-6} (iii) we get $K \in \mathcal{K}_{A^{1/2}}(\mathcal{H}).$ Since $\widetilde{T+K}=\widetilde{T}+\widetilde{K}$ is normal so from Lemma \ref{lemma1-2} (ii) and Theorem \ref{L1} it follows that $T+K$ is $A$-normal. Now, from Lemma \ref{lemma1-2} (i) and Lemma \ref{lemma3-6} (iv) we have $\overline{W_A(T+K)}=W_A(T+K).$ Also by applying Lemma \ref{lemma3-6} (ii), we get $\|K\|_A<\epsilon.$ This completes the proof.
\end{proof}

The following example shows that closedness of $R(A)$ is not necessary for Theorem \ref{T3}.


\begin{example}\label{example1}
Considering the  operators mentioned in  Example \ref{ex1} it is easy to verify that $T$ is $A$-normal 
 and by applying Theorem \ref{T12} we get 
\begin{align*}
\overline{W_A(T)}=conv\{\lambda_n : n \in \mathbb{N}\cup\{0\}\}. 
\end{align*}
Now, we have
\begin{eqnarray*}
	W_A(T)
	= \left\{ \sum_{n=1}^{\infty} \frac{\lambda_n|x_n|^2}{n}:  \sum_{n=1}^{\infty} |x_n|^2 <\infty~~\mbox{and}~~ \sum_{n=1}^{\infty} \frac{|x_n|^2}{n}=1 \right\}.
\end{eqnarray*}
Clearly, for all $n \in \mathbb{N},$ $\lambda_n \in W_A(T)$  and  $\lim_{n \to \infty}\lambda_n=\lambda_0 \in \overline{W_A(T)}.$ Suppose that $\lambda_0 \in W_A(T).$ Thus, $\lambda_0=\sum_{n=1}^{\infty} \lambda_n\frac{|x_n|^2}{n}$ with $\sum_{n=1}^{\infty}\frac{|x_n|^2}{n}=1.$ This implies that $\sum_{n=1}^{\infty} (\lambda_0-\lambda_n)\frac{|x_n|^2}{n}=0,$ and so $\sum_{n=1}^{\infty} Re(\lambda_0-\lambda_n)\frac{|x_n|^2}{n}=0.$ This leads to a contradiction, as $\sum_{n=1}^{\infty}|x_n|^2=1$ and $Re(\lambda_0-\lambda_n)>0$ for all $ n \in \mathbb{N}.$ Thus for this example $W_A(T)$ is not closed.

Given $\epsilon>0$ there exists $n_0 \in \mathbb{N}$ such that $|\lambda_n-\lambda_0|<\epsilon$ for all $n \geq n_0.$ Now, we define the bounded linear operator $ K : \ell_2 \to \ell_2$ by
 $$Ke_n=\begin{cases}
 	(\lambda_0-\lambda_{n_0})e_{n_0}~\mbox{for $n=n_0$}\\
 	\,\,\,\,\,\,\,\,\,0~~~\,\,\,\,\,\,\,\,\,\,\,\,\,\,\,\,\,\,\,\,\,~\mbox{for $n\neq n_0$}.
 \end{cases}$$
 Then 
 $$(T+K)e_n=\begin{cases}
 	\lambda_0e_{n_0}~\,\mbox{for $n=n_0$}\\
 	\lambda_ne_{n}~\,\,\,\mbox{for $n\neq n_0$}.
 \end{cases}$$
It is easy to observe that $\|K\|_A=|\lambda_0-\lambda_{n_0}|<\epsilon$ and $K \in \mathcal{K}_{A^{1/2}}(\ell_2).$ Proceeding similarly as in Example \ref{ex1} we can infer that $T+K$ is $A$-normal and $\overline{W_A(T+K)}=conv\{\lambda_0,\lambda_n : n \in \mathbb{N}\setminus \{n_0\}\}.$ Again
\begin{eqnarray*}
	W_A(T+K)
	= \left\{\frac{\lambda_0|x_{n_0}|^2}{n_0}+ \sum_{n=1,n\neq n_0}^{\infty} \frac{\lambda_n|x_n|^2}{n}:  \sum_{n=1}^{\infty} |x_n|^2 <\infty~~\mbox{and}~~ \sum_{n=1}^{\infty} \frac{|x_n|^2}{n}=1 \right\}.
\end{eqnarray*}
Clearly, $\{\lambda_0,\lambda_n : n \in \mathbb{N}\setminus \{n_0\}\} \subseteq W_A(T+K)$ and so
${W_A(T+K)}=conv\{\lambda_0,\lambda_n : n \in \mathbb{N}\setminus \{n_0\}\}.$ Hence $\overline{W_A(T+K)}=W_A(T+K).$ 
\end{example}

\begin{remark}
	The question that remains to be answered is that whether Theorems \ref{T1}, \ref{T2} and \ref{T3} remain valid without the assumption that $R(A)$ is closed.
\end{remark}

\bibliographystyle{amsplain}

\begin{thebibliography}{99}
	
	
	\bibitem{arias 1} M.L. Arias, G. Corach and M.C. Gonzalez, Lifting properties in operator ranges, Acta Sci. Math. (Szeged), 75:3-4 (2009), 635--653.
	
	\bibitem{a2} M.L. Arias, G. Corach and M.C. Gonzalez, Partial isometries in semi-Hilbertian spaces, Linear Algebra Appl., 428 (2008), no. 7, 1460--1475.
	
	\bibitem{a3} M.L. Arias, G. Corach and M.C. Gonzalez, Metric properties of projections in semi-Hilbertian spaces,
	Integral Equations Operator Theory, 62 (2008), 11--28.
	
	\bibitem{Atkinson}  F.V. Atkinson, The normal solvability of linear equations in normed spaces, Mat. Sb., 28 (70) (1951), 3--14.
	
	\bibitem{baklouti_OAM_2023}  H. Baklouti and M. Mabrouk, A note on the $A$-spectrum of $A$-bounded operators, Oper. Matrices, 17 (2023), 599--611.
	
    \bibitem{baklouti_BJMA_2022}  H. Baklouti and S. Namouri, Spectral analysis of bounded operators on semi-Hilbertian spaces, Banach. J. Math. Anal., 16 (2022), no. 1, Paper No. 12, 17pp.
    
    
    \bibitem{bak}  H. Baklouti, K. Feki and O.A.M. Sid Ahmed, Joint numerical ranges of operators in semi-Hilbertian spaces, Linear Algebra Appl., 555 (2018), 266--284.
  
  \bibitem{Berberian_DJM_1981} S.K. Berberian, The numerical range of a normal operator, Duke Math. J.,31 (1964), 479--483.
  
	
	\bibitem{BKPS_LAA_23} P. Bhunia, F. Kittaneh, K. Paul and A. Sen, Anderson's theorem and A-spectral radius bounds for semi-Hilbertian space operators, Linear Algebra Appl., 657 (2023), 147--162.
	
	\bibitem{BNK_RM_2021} P. Bhunia, R.K. Nayak and K. Paul, Improvement of  $A$-numerical radius inequalities of semi-Hilbertian space operators, Results Math., 76 (2021), no. 3, Paper No. 120, 10 pp.
	
	\bibitem{BSS_JMA_2018} R. Birbonshi, I.M. Spitkovsky and P.D. Srivastava, A note on Anderson’s theorem in the infinite-dimensional setting, J. Math. Anal. Appl., 461 (2018), no. 1, 349--353.
	
	\bibitem{Bourin_JOT_2003} J.-C. Bourin, Compressions and pinchings, J. Oper. Theory, 50(2) (2003), 211--220.
	
	
	\bibitem{branges1}  L. de Branges and J. Rovnyak, Square Summable Power Series, Holt, Rinehert and Winston, New York, (1966).
	
	\bibitem{dog1} R.G. Douglas, On majorization, factorization and range inclusion of operators in Hilbert space, Proc. Amer. Math. Soc., 17 (1966), 413--416.
	
	\bibitem{DWMAMS} M. Dritschel and H. J. Woerdeman, Model theory and linear extreme points in the numerical radius unit ball, Mem. Amer. Math. Soc., 129(1997), no.615, viii+62 pp.
	
	\bibitem{Engl_JMAA_1981} H.W. Engl and M.Z. Nashed, New extremal characterizations of generalized inverses of linear operators, J. Math. Anal. Appl., 82 (1981), no. 2, 566--586.
	
	\bibitem{feki_BJOM_2022} K. Feki, Some $A$-spectral radius inequalities for $A$-bounded Hilbert space operators, Banach J. Math. Anal., 16 (2022), no. 2, Paper No. 31, 22 pp.
	
	\bibitem{feki_AOFA_2020} K. Feki, Spectral radius of semi-Hilbertian space operators and its applications, Ann. Funct. Anal., 11 (2020), no. 4, 929–946.
	
   \bibitem{feki_LAA_2020}	K. Feki, On tuples of commuting operators in positive semidefinite inner product spaces, Linear Algebra Appl., 603 (2020), 313--328. 
   
   \bibitem{FSW_1972} P.A. Fillmore, J.G. Stampfli and J.P. Williams, On the essential numerical range, the essential spectrum, and a problem of Halmos, Acta Sci. Math. (Szeged), 33 (1972), 179--192.
   
   \bibitem{Garcia_TAMS_2006} S.R. Garcia and M. Putinar, Complex symmetric operators and applications, Trans. Amer. Math. Soc., 358 (2006), 1285--1315.
   
   
   \bibitem{Gau_PAMS} H.-L. Gau and P.Y. Wu, Anderson’s theorem for compact operators, Proc. Amer. Math. Soc., 134 (2006), no. 11, 3159--3162.
   
  \bibitem{Gustafson_book} K.E. Gustafson and D.K.M. Rao, Numerical Range. The Field of Values of Linear Operators and Matrices, Springer-Verlag, New York, (1997).
  
   \bibitem{Ji_oam_18}  Y. Ji and B. Liang, On operators with closed numerical ranges, Ann. Funct. Anal., 9(2) (2018),  233--245.
  
  \bibitem{KZ_JCAM_2023} F. Kittaneh and A. Zamani, Bounds for $\mathbb{A}$-numerical radius based on an extension of  $A$-Buzano inequality,
  J. Comput. Appl. Math., 426 (2023), Paper No. 115070, 14 pp.
  
   
   \bibitem{liang_JMAA_2023} B. Liang, Complex symmetric operators with closed numerical range, J. Math. Anal. Appl., 520 (2023), no. 2, Paper No. 126898, 8pp.
   
   \bibitem{maj_LAMA} W. Majdak, N.A. Secelean and L. Suciu, Ergodic properties of operators in some semi-Hilbertian spaces, Linear Multilinear Algebra, 61(2) (2013), 139--159.
   
   \bibitem{MXZ_LAA_2020} M.S. Moslehian, Q. Xu and A. Zamani, Seminorm and numerical radius inequalities of operators in semi-Hilbertian spaces, Linear Algebra Appl., 591 (2020), 299--321.
   
   
   \bibitem{Penr55} R. Penrose, A generalized inverse for matrices, Proc. Cambridge Philos. Soc.,  51 (1955), 406--413.
   
   \bibitem{Saddi} A. Saddi,  $A$-normal operators in semi-Hilbertian spaces, Aust. J. Math. Anal. Appl., 9(1) (2012),  1--12.
   
   
   
   \bibitem{TYLAA} B.-S. Tam and S. Yang, On matrices whose numerical ranges have circular or weak circular symmetry, Linear Algebra Appl., 302/303 (1998), 193--221.
   
  \bibitem{wu_gau_book} P.Y. Wu and H.-L. Gau, Numerical ranges of Hilbert space operators, Encyclopedia Math. Appl., 179, Cambridge University Press, Cambridge, (2021).
  
 
   
   \bibitem{zamani} A. Zamani, $A$-numerical radius inequalities for semi-Hilbertian space operators, Linear Algebra Appl., 578 (2019), 159--183.
   
  \bibitem{zhu_BJMA_2015} S. Zhu, Approximate unitary equivalence of normaloid type operators, Banach J. Math. Anal., 9(3) (2015), 173--193.
  












 
 





   
\end{thebibliography}

\end{document}